\documentclass[a4paper,12pt]{article}

\usepackage{amsfonts}
\usepackage{amscd,color}
\usepackage{amsmath,amsfonts,amssymb,amscd}
\usepackage{indentfirst,graphicx,epsfig}
\usepackage{graphicx}
\input{epsf}
\usepackage{graphicx}
\usepackage{epstopdf}
\usepackage{caption}

\newtheorem{thm}{Theorem}[section]

\newtheorem{lem}[thm]{Lemma}

\newtheorem{cor}[thm]{Corollary}
\newtheorem{pro}[thm]{Proposition}
\newenvironment {proof} {\noindent{\em Proof.}}{\hspace*{\fill}$\Box$\par\vspace{4mm}}
\newcommand{\ml}{l\kern-0.55mm\char39\kern-0.3mm}

\title{\textbf{More on limited packings in graphs\footnote{Supported by NSFC No.11531011.}}}
\author{{\small Xuqing Bai, Hong Chang, Xueliang Li } \\
{\small  Center for Combinatorics and LPMC}\\
{\small Nankai University, Tianjin 300071, P.R. China}\\
{\small Email: baixuqing0@163.com, changh@mail.nankai.edu.cn, lxl@nankai.edu.cn}\\
}
\date{}
\begin{document}
\maketitle
\begin{abstract}
A set $B$ of vertices in a graph $G$ is called a \emph{$k$-limited packing}
if for each vertex $v$ of $G$, its closed neighbourhood has at most $k$
vertices in $B$. The \emph{$k$-limited packing number} of a graph $G$,
denoted by $L_k(G)$, is the largest number of vertices in a $k$-limited packing
in $G$. The concept of the $k$-limited packing of a graph was introduced by
Gallant et al., which is a generalization of the well-known
packing of a graph. In this paper, we present some tight bounds for
the $k$-limited packing number of a graph in terms of its order, diameter,
girth, and maximum degree, respectively. As a result, we obtain the tight 
Nordhaus-Gaddum-type result of this parameter for general $k$. At last,
we investigate the relationship among the open packing number, the packing
number and $2$-limited packing number of trees.


\end{abstract}

\section{Introduction}

All graphs in this paper are undirected, simple and nontrivial.
We follow \cite{BM} for graph theoretical notation and terminology
not described here. Let $G$ be a graph, we use $V(G), E(G), diam(G),
\Delta(G)$ and $\delta(G)$ to denote the vertex set, edge set,
diameter, maximum degree, and minimum degree of $G$, respectively.
Take a vertex $v\in V(G)$, the \emph{open neighbourhood} of $v$ is
defined as the set of all vertices adjacent to $v$ in $G$, the set
$N[v] = \{v\}\cup N(v)$ is called \emph {the closed neighbourhood}
of $v$ in $G$.

A set $D$ of vertices in a graph $G$ is called a \emph{dominating set}
if each vertex in $V(G)\setminus D$ has at least one neighbour in $D$.
The \emph{domination number} $\gamma(G)$ of a graph $G$ is the minimum
cardinality of a dominating set in $G$. The theory of dominating sets,
introduced formally by Ore \cite{O} and Berge \cite{Berge}, has been
the subject of many recent papers due to its practical and theoretical
interest. For more information on domination topics we refer to the books
\cite{HHS, HHS1}. A domination set $D$ of a graph $G$ is called a
\emph{total dominating set} if $G[D]$ has no isolated vertex, and
the minimum cardinality of a total dominating set in $G$ is called
the \emph{total domination number} of $G$, denoted by $\gamma_t(G)$.
Total domination in graphs was introduced by Cockayne, Dawes,
and Hedetniemi \cite{CDH}, and has been well studied (see, for example, \cite{FH,FHMP,HY,Y}).

On the other side, the \emph{open packing} of a graph $G$ is a set
$S$ of vertices in $G$ such that for each vertex $v$ of $G$,
$|N(v)\cap S|\leq 1$. The \emph{open packing number} of a graph $G$,
denoted by $\rho^{0}(G)$, is the maximum cardinality among all open
packings in $G$. The open packing of a graph has been studied in \cite{H,HSE}.

The well-known \emph{packing} (\emph{$2$-packing}) of a graph $G$ is
a set $B$ of vertices in $G$ such that $|N[v]\cap B|\leq 1$ for each
vertex $v$ of $G$. The \emph{packing number} $\rho(G)$ of a graph $G$
is the maximum cardinality of a packing in $G$. The packing of a graph
has been well studied in the literature \cite{B,C,MM,TV}. Dominating
sets and packings of graphs are two good models for many utility location
problems in operations research. But the corresponding problems have a
very different nature: the former is a minimization problem (dominating sets)
to satisfy some reliability requirements, the latter is a maximization problem
not to break some (security) constraints. Consider the following scenarios:

Network security: A set of sensors is to be deployed to covertly monitor
a facility. Too many sensors close to any given location in the facility
can be detected. Where should the sensors be placed so that the total
number of sensors deployed is maximized?

Market Saturation: A fast food franchise is moving into a new city. Market
analysis shows that each outlet draws customers from both its immediate
city block and from nearby city blocks. However it is also known that a
given city block cannot support too many outlets nearby. Where should the
outlets be placed?

Codes: Information is to be transmitted between two interested parties.
This data is first represented by bit strings (codewords) of length $n$.
It is desirable to be able to use as many of these $2^n$ strings as possible. However, if a single bit of a codeword is altered during transmission,
we should still be able to recover the piece of data correctly by employing a ``nearest neighbour" decoding algorithm. How many code words can be used as a
function of $n$?

A graph model of these scenarios might maximize the size of a vertex subset
subject to the constraint that no vertex in the graph is near too many of the selected vertices.

Motivated by the packing of graphs, Gallant et al. relaxed the constraints
and introduced the concept of the \emph{$k$-limited packing} in graphs in \cite{GGHR}.  A set $B$ of vertices in a graph $G$ is called a
\emph{$k$-limited packing} if for each vertex $v$ of $G$,
$|N[v]\cap B|\leq k$. The \emph{$k$-limited packing number} of a graph $G$,
denoted by $L_k(G)$, is the largest number of vertices in a $k$-limited
packing set in $G$. It is clear that $L_1(G) = \rho(G)$. The problem of
finding a $1$-limited packing of maximum size for a graph is shown to be
NP-complete in \cite{HS0}. In \cite{DLN}, it is shown that the problem
of finding a maximum size $k$-limited packing is NP-complete even in split
or bipartite graphs. For more results on $k$-limited packings of graphs,
we refer to \cite{BBG,GZ,GGHR,LN}.

The remainder of this paper will be organized as follows. In Section $2$,
we give the technical preliminaries, including notations and relevant
known results on open packings and $k$-limited packings of graphs.
In Section $3$, we present some tight bounds for the $k$-limited packing
number of a graph in terms of its order, diameter, girth, and maximum
degree, respectively. Based on them, we obtain the tight Nordhaus-Gaddum-type
result of this parameter for general $k$. In Section $4$, we focus on
the $2$-limited packing number of graphs, including trees and graphs
with diameter two. And we get the better upper bound of the $2$-limited
packing number of graphs with large diameter. In Section $5$, we investigate
the relationship among the open packing number, the $1$-limited packing number
and the $2$-limited packing number of trees.

\section{Preliminaries}

The notation we use is mostly standard. For $B \subseteq V(G)$,
let $\overline{B}=V(G)\backslash B$, and $G[B]$ denote the subgraph
of $G$ induced by $B$. Given $t$ graphs $G_1,\ldots, G_t$, the
\emph{union} of $G_1,\ldots,G_t$, denoted by $G_1\cup \cdots \cup G_t$,
is the graph with vertex set $V(G_1)\cup \cdots \cup V(G_t)$ and edge set
$E(G_1)\cup \cdots \cup E(G_t)$. In particular, let $tG$ denote
the vertex-disjoint \emph{union} of $G_1,\ldots G_t$ for $G_1= \cdots =G_t=G$.

We next state some relevant known results on $k$-limited packings
of graphs, which will be needed later.

\begin{lem}{\upshape\cite{HS}}\label{lem15}
Let $G$ be a graph of order at least $3$. Then $\rho^{0}(G) = 1$
if and only if $diam(G)\leq 2$ and every edge of $G$ lies on a triangle.
\end{lem}

\begin{lem}{\upshape\cite{HS}}\label{lem17}
If $G$ is a graph of diameter $2$, then $\rho^{0}(G)\leq 2$.
\end{lem}

\noindent\textbf{Remark 1.}
It is clear that graphs with diameter $1$, which are exactly complete
graphs, have opening packing number at most $2$. Thus, if $G$ is a graph of
diameter at most $2$, then $\rho^{0}(G)\leq 2$.

\begin{lem}{\upshape\cite{R}}\label{lem33}
If $T$ is any tree of order at least $2$, then $\rho^{0}(T)=\gamma_t(T)$.
\end{lem}

\begin{lem}{\upshape\cite{MS}}\label{lem16}
For any graph $G$, $L_1(G) = 1$ if and only if $diam(G)\leq2$.
\end{lem}

\begin{lem}{\upshape\cite{BBG}}\label{lem31}
For any graph $G$ of order $n$,
$L_1(G)\geq \frac{n}{\Delta(G)^2+1}.$
\end{lem}

\begin{lem}{\upshape\cite{MM}}\label{lem34}
For any tree T, $L_1(T)=\gamma(T).$
\end{lem}

\begin{lem}{\upshape\cite{MS}}\label{lem7}
For any connected graph $G$ and integer $k\in \{1,2\}$,
$$ L_k(G)\geq \lceil \frac{k+kdiam(G)}{3}\rceil .$$
\end{lem}

\begin{lem}{\upshape\cite{S}} \label{lem2}
Let $G$ be a graph of order $n$. Then $L_2(G)+L_2(\overline{G})\leq n+2$,
and this bound is tight.
\end{lem}

Since a $k$-limited packing of a graph is also a $(k+1)$-limited
packing, we immediately obtain the following inequalities:
$L_1(G)\leq L_2(G)\leq\cdots\leq L_{k}(G)\leq L_{k+1}(G)\leq\cdots$.
Furthermore, the authors obtained the stronger result in \cite{MSH}.

\begin{lem}{\upshape\cite{MSH}}\label{lem5}
Let $G$ be a connected graph of order $n$ and $k\leq \Delta(G)$. Then
$L_{k+1}(G)\geq L_k(G)+1$. Moreover, $L_{k}(G)\geq L_1(G)+k-1$, and
this bound is tight.
\end{lem}

\noindent\textbf{Remark 2.}
Based on the proof of Lemma \ref{lem5} in \cite{MSH}, the condition of
the connectivity of $G$ in Lemma \ref{lem5} can be deleted.

\begin{lem}{\upshape\cite{BBG}}\label{lem10}
For any graph $G$ of order $n$, $L_k(G)\leq \frac{kn}{\delta(G)+1}.$
\end{lem}

In the sequel, let $P_n$, $C_n$, $K_n$, and $K_{s,t}$ denote the path
of order $n$, cycle of order $n$, complete graph of order $n$,
and complete bipartite graph of order $s+t$, respectively.
It is clear that $L_k(P_n)=L_k(C_n)=n$ for $k\geq 3$.

\begin{lem}{\upshape \cite{GGHR}}\label{lem3}
Let $m, n, k\in \mathbb{N}$. Then

$(i)$ $L_k(P_n)=\lceil \frac{kn}{3}\rceil $ for $k=1,2$,

$(ii)$ $L_k(C_n)=\lfloor\frac{kn}{3}\rfloor$ for $k=1,2$ and $n\geq 3$,

$(iii)$ $L_k(K_{n})=min\{k,n\},$

$(iv)$ $L_k(K_{m,n})=
\begin{cases}
1& \text{if $k=1$},\\
min\{k-1,m\}+min\{k-1,n\}& \text{if $k > 1$}.
\end{cases}$
\end{lem}

\begin{lem}{\upshape\cite{GGHR}}\label{lem4}
If $G$ is a graph, then $L_k(G)\leq k\gamma(G)$. Furthermore, the
equality holds if and only if for any maximum $k$-limited packing $B$ in
$G$ and any minimum dominating set $D$ in $G$ both the following hold:

$(i)$ For any $b\in B$ we have $\mid N[b]\cap D\mid=1$,

$(ii)$ For any $d\in D$ we have $\mid N[d]\cap B\mid=k$.
\end{lem}

It is worth mentioning that we generalize the results of Lemma \ref{lem16},
Lemma \ref{lem7} and lemma \ref{lem2} to general $k$-limited packing
parameter of graphs, and characterize all the trees $T$ satisfying
$L_{2}(T)= L_1(T)+1$ in Lemma \ref{lem5} later, which are parts of our job.

\section{$k$-limited packing}

In this section we present some tight bounds for the $k$-limited
packing number of a graph in terms of its order, diameter, girth,
and maximum degree, respectively. As a result, we obtain the
tight Nordhaus-Gaddum-type result for this parameter.

It is clear to obtain the following result.

\begin{pro}\label{prop1}
If $G$ is a graph of order $n$ with $n\leq k$, then $L_k(G)=n$.
\end{pro}

\noindent\textbf{Remark 3.}
Actually, the above condition that $n\leq k$ can be weakened to
$\Delta(G)+1 \leq k$. So, we only need to consider the $k$-limited
packing number for graphs $G$ with $\Delta(G)\geq k$.

\begin{pro}\label{prop2}
If $G$ is a graph of order $k+1$, then
\begin{equation*}
L_k(G)=
\left\{
  \begin{array}{ll}
    k & \hbox{ if  $\Delta(G)=k$,}\\
    k+1 & \hbox{ otherwise.} \\
    \end{array}
\right.
\end{equation*}
\end{pro}

\begin{proof}
Let $G$ be a graph of order $k+1$. Then $k \leq L_k(G)\leq k+1$. Let
$\Delta(G)=k$. Assume to the contrary that $L_k(G)= k+1$. It is obtained
that $V(G)$ is the unique maximum $k$-limited packing of $G$. Let $v_0$
be a vertex with maximum degree $k$ in $G$. Then $|N[v_0]\cap V(G)|=k+1$,
which is a contradiction. Thus, $L_k(G)= k$. It remains to show the other
case. Let $\Delta(G)\leq k-1$. Obviously, $V(G)$ is a $k$-limited packing
of $G$, it follows that $L_k(G)=k+1$.
\end{proof}

For a given graph $G$ of order less than $k+2$, we can determine its
$k$-limited packing number by Proposition \ref{prop1} and Proposition
\ref{prop2}. So we are concerned with graphs of order at least $k+2$ in
the following.

\begin{pro}\label{prop3}
If $G$ is a graph of order at least $k+2$, then
$L_k(G)\geq k$.
\end{pro}

The following result is a generalization of Lemma \ref{lem16}.

\begin{thm}\label{th1}
Let $G$ be a graph of order $n$. Then $L_k(G)=k$ if and only if
one of the following conditions holds:

$(i)$ $n=k$,

$(ii)$ $\Delta(G)=k$, where $n=k+1$,

$(iii)$ for each $(k+1)$-subset $X$ of $V(G)$, $G[X]$ has maximum degree
$k$ or the $k+1$ vertices of $X$ have a common neighbour, where $n\geq k+2$.
\end{thm}

\begin{proof}
The statement holds for $n\leq k+1$ by Proposition \ref{prop1} and
Proposition \ref{prop2}, thus we may assume that $n\geq k+2$ in the
following. Notice that $k \leq L_k(G)\leq n$ for $n\geq k+2$. Let $G$
be a graph of order $n$ such that for each $(k+1)$-subset $X$ of $V(G)$,
$G[X]$ has maximum degree $k$ or the $k+1$ vertices of $X$ have a
common neighbour. Assume that $G$ has a $k$-limited packing $B$ with
at least $k+1$ vertices. Let $X$ be a $(k+1)$-subset of $B$.
Obviously, $X$ is also a $(k+1)$-subset of $V(G)$. If $G[X]$ has a
vertex $v_0$ with degree $k$, then $|N[v_0]\cap B|\geq k+1$, which
is a contradiction. If the $k+1$ vertices of $X$ have a common neighbour
$a$, then $|N[a]\cap B|\geq k+1$, which is also a contradiction.
Thus, $L_k(G)= k$.

It remains to show the converse. Let $G$ be a graph of order $n$ such
that $L_k(G)= k$. Assume that there exists a $(k+1)$-subset $X_0$ of $V(G)$
such that $G[X_0]$ has maximum degree at most $k-1$ and each vertex outside
$X_0$ is adjacent to at most $k$ vertices in $X_0$. It follows that $X_0$ is
a $k$-limited packing of $G$, which implies that $L_k(G)\geq k+1$,
a contradiction. Therefore, if $G$ is a graph of order at least $k+2$ with
$L_k(G)= k$, then for each $(k+1)$-subset $X$ of $V(G)$, $G[X]$ has maximum
degree $k$ or the $k+1$ vertices of $X$ have a common neighbour.
\end{proof}

The following result is an immediate and obvious corollary of the
above theorem.

\begin{cor}\label{coro2}
Let $G$ be a graph of order at least $k+1$ such that $L_k(G)=k$.
Then $diam(G)\leq 2$.
\end{cor}

Next, we present a lower bound of the $k$-limited packing number of a
graph in terms of it diameter for $k\geq 3$, which is a generalization
of Lemma \ref{lem7}.

\begin{thm}\label{th5}
Let $G$ be a connected graph and $ \Delta(G)\geq k\geq 3$. Then
$L_k(G)\geq diam(G)+k-2$. Moreover, the lower bound is tight.
\end{thm}

\begin{proof}
Let $P=v_1v_2\cdots v_{diam(G)+1}$ be a path of length $diam(G)$
in $G$. Obviously, for each vertex $v_i$ on $P$,
$|N[v_i]\cap V(P)|\leq 3$. We claim that $V(P)$ is a $3$-limited
packing in $G$. Assume to the contrary that there exists a vertex
$u$ outside $P$ such that $|N(u)\cap V(P)|\geq 4$. Let
$N(u)\cap V(P)=\{v_{i_1},\ldots,v_{i_d}\}$
with $i_1\leq \cdots \leq i_d$ and $d\geq 4$. Then
$P'=v_1\cdots v_{i_1}uv_{i_d}\cdots v_{diam(G)+1}$ is a path between
$v_1$ and $v_{diam(G)+1}$, whose length is less than $diam(G)$, a
contradiction. Thus, $L_3(G)\geq |V(P)|=diam(G)+1$. And by
Lemma \ref{lem5}, we have
$L_k(G)\geq L_3(G)+k-3\geq diam(G)+1+k-3=diam(G)+k-2.$

Corollary \ref{coro2} shows that non-complete graphs $G$ of order at least
$k+1$ with $L_k(G)=k$ are ones satisfying $L_k(G)=diam(G)+k-2$.
\end{proof}

Recall that \emph{the girth} of a graph $G$ is the length of a shortest cycle
in $G$, denoted by $g(G)$.

\begin{thm}\label{th7}
If $G$ is a graph with girth $g(G)$, then
$L_1(G)\geq \lfloor\frac{g(G)}{3}\rfloor$.
Moreover, the lower bound is tight.
\end{thm}

\begin{proof}
Let $G$ be a graph and $C$ be a cycle of length $g(G)$ in $G$.
The statement is evidently true for $g(G)\leq 4$. Thus, we only
need to consider the case when $g(G)\geq 5$. Let $B$ be a maximum
$1$-limited packing of $C.$ Then
$|B|=L_1(C)=\lfloor\frac{g(G)}{3}\rfloor$ by Lemma \ref{lem3}.
Next we will show that $B$ is also a $1$-limited packing of $G$,
which implies that $L_1(G)\geq \lfloor\frac{g(G)}{3}\rfloor$.
It is sufficient to show that each vertex $v$ outside $C$ has at most
one neighbour on $C$. Assume to the contrary that there is a vertex
$v_0$ outside $C$ that is adjacent to two vertices, say $x,$ $y$,
on $C$. Let $P$ be the shortest path between $x$ and $y$ on $C$.
If $|V(P)|\leq 3$, then $x$, $v_0$ and $y$ are on either a $C_3$ or
$C_4$ in $G$, contradicting $g(G)\geq 5.$ Now we may assume that
$|V(P)|\geq 4$. Let $C'$ be the cycle obtained from $C$ replacing $P$
by the path $xv_0y$. Then the length of $C'$ is less than $g(G)$, which
is a contradiction. Thus, each vertex outside $C$ has at most one neighbour
on $C$.  Furthermore, cycles are graphs $G$ with
$L_1(G)= \lfloor\frac{g(G)}{3}\rfloor$ by Lemma \ref{lem3}. The proof is complete.
\end{proof}

Next, we give the lower bound of the $k$-limited packing number of a
graph with respect to its girth for general $k\geq 2$.

\begin{thm}\label{th2}
If $G$ is a graph with girth $g(G)$, then
$L_2(G)\geq \lfloor\frac{2g(G)}{3}\rfloor$ and
$L_k(G)\geq g(G)+k-3$ for $ \Delta(G)\geq k \geq 3$.
Moreover, the lower bounds are tight.
\end{thm}

\begin{proof}
The statement trivially holds for $g(G)\leq 3$. Thus, we may assume that
$g(G)\geq 4$ in the following. Let $C$ be a cycle of length $g(G)$ in $G$.
We first present the following claim.

\textbf{Claim 1.} Each vertex outside $C$ has at most two neighbours on $C$.

\noindent\textbf{Proof of Claim 1:}
Suppose to the contrary that there is a vertex $v$ outside $C$ such
that $v$ is adjacent to three vertices, say $x,$ $y,$ $z$, on $C$.
Let $P$ be the shortest path containing $x,$ $y,$ $z$ on $C$ such
that the end vertices of $P$ are contained in $\{x,y,z\}.$  Without
loss of generality, assume that $x$, $y$ are the end vertices of $P.$
Obviously, $|V(P)|\geq 3.$ Suppose that $|V(P)|=3$. It follows that
$z$ is adjacent to both $x$ and $y$. Therefore,  $G[\{x,z,v\}]$ is exactly
$C_3$, which contradicts to that $g(G)\geq 4.$ Suppose that $|V(P)|\geq 4$.
Let $C'$ be the cycle obtained from $C$ replacing $P$ by the path $xvy$.
Then the length of $C'$ is less than $g(G)$, which is a contradiction.
Thus, each vertex outside $C$ has at most two neighbours on $C$.

Let $B$ be a maximum $2$-limited packing of $C.$
Then $|B|=L_2(C)=\lfloor\frac{2g(G)}{3}\rfloor$ by Lemma \ref{lem3}.
By Claim $1$, we obtain that $B$ is also a $2$-limited packing of $G$.
Thus, $L_2(G)\geq|B| =\lfloor\frac{2g(G)}{3}\rfloor$. Moreover,
cycles are graphs $G$ with $L_2(G)= \lfloor\frac{2g(G)}{3}\rfloor$
by Lemma \ref{lem3}.

Observe that $V(C)$ is a maximum $3$-limited packing of $C$.
And by Claim $1$, it is known that $V(C)$ is also a $3$-limited packing
of $G$, which implies that $L_3(G)\geq g(G)$. It follows from
Lemma \ref{lem5} and Remark $2$ that $L_k(G)\geq L_3(G)+k-3\geq g(G)+k-3$
for $k\geq 3$. Furthermore, graphs $G$ of order at least $k+1$ with triangles, satisfying $L_k(G)=k$, have the property that $L_k(G)=g(G)+k-3$.
\end{proof}

Next, we turn to study the upper bound of the $k$-limited packing
number of a graph.

\begin{thm}\label{th3}
If $G$ is a graph of order $n$,
then $L_k(G)\leq n+k-1-\Delta(G)$.
\end{thm}

\begin{proof}
Let $v_0$ be a vertex of maximum degree $\Delta(G)$ in $G$.
If $k\geq \Delta(G)+1$, then it is clear that $V(G)$ is a
$k$-limited packing of $G$, and hence $L_k(G)=n\leq n+k-1-\Delta(G)$.
Thus, we may assume that $k< \Delta(G)+1$ in the following. Let $B$
be a maximum $k$-limited packing of $G$. Since $|N[v_0]\cap B|\leq k$,
it follows  that there exist at least $\Delta(G)+1-k$ vertices in
$N[v_0]\setminus B$, which means that $|\overline{B}|\geq \Delta(G)+1-k$.
Thus, $L_k(G)=|B|=n-|\overline{B}|\leq n-(\Delta(G)+1-k)=n+k-1-\Delta(G)$.
\end{proof}

We define the graph class $\mathcal{G}$ consisting of all graphs $G$
constructed as follows. Let $G$ be a graph of order $n$ such that
$V(G)=A_0\cup B_0$ has the following properties:

$(i)$ $|A_0\cap B_0|=2$,

$(ii)$ $G[A_0]$ has a spanning star, and each component of $G[B_0]$
is $K_1$ or $K_2$,

$(iii)$ for each vertex $v\in \overline{B_0}$, $|N(v)\cap B_0|\leq 2$.

The following result shows that $\mathcal{G}$ is the set of
all graphs $G$ of order $n$ with $L_2(G)= n+1-\Delta(G)$.

\begin{cor}\label{cor14}
If $G$ is a graph of order $n$, then $L_2(G)\leq n+1-\Delta(G)$.
Moreover, $L_2(G)= n+1-\Delta(G)$ if and only if $G \in \mathcal{G}$.
\end{cor}

\begin{proof}
We first restate the proof for Theorem \ref{th3}. Let $B$ be a maximum
$2$-limited packing in $G$. Obviously, each component of $G[B]$ is
$K_1$ or $K_2$, and $|N[v]\cap B|\leq 2$ for each vertex $v$ of $G$.
Let $v_0$ be a vertex of maximum degree $\Delta(G)$.
Since $|N[v_0]\cap B|\leq 2$, it follows that there exist at least
$\Delta(G)-1$ vertices in $N[v_0]\setminus B$. Thus,
$L_2(G)=|B|=n-|\overline{B}|\leq n-(\Delta(G)-1)=n+1-\Delta(G)$.
Let $G$ be a graph of order $n$ such that $L_2(G)=n+1-\Delta(G)$.
It is easily obtained that $G$ has the following properties:

$(P1)$ $|N[v_0]\cap B|=2$,

$(P2)$ $V(G)\setminus N[v_0]\subset B$.

By the above argument, we have $G \in \mathcal{G}$ with $N[v_0]=A_0$
and $B=B_0$. It remains to show the converse. Suppose that
$G \in \mathcal{G}$. It is sufficient to show that
$L_2(G)\geq n+1-\Delta(G)$. Let $A_0\cap B_0=\{v_p,v_q\}$ and
$|A_0|=t+1$, where $v_0$ is a vertex of degree $t$ in $G[A_0]$.
Observe that $d(v_0)\geq t$. Furthermore, we obtain the following claim.

\textbf{Claim 1.} $\Delta(G)=t$.

\noindent\textbf{Proof of Claim 1:} Since each vertex in $B_0$ has
degree at most $1$ in $G[B_0]$, it follows that each of $v_p,v_q$
is adjacent to at most one vertex in $B_0$. On the other hand, each
of $v_p,v_q$ is adjacent to at most $t-1$ vertices in
$A_0\setminus\{v_p,v_q\}$. Thus, $d(v_p)\leq t$ and $d(v_q)\leq t$.
For each vertex $v$ in $A_0\setminus\{v_p,v_q\}$, $v$ is adjacent to
at most $t-2$ vertices in $A_0\setminus\{v,v_p,v_q\}$ and at most two
vertices in $B_0$, thus $d(v)\leq t$ for each vertex $v$ in $A_0\setminus\{v_p,v_q\}$. For each vertex $u$ in $B_0\setminus\{v_p,v_q\}$,
$u$ is adjacent to at most $t-1$ vertices in $A_0\setminus\{v_p,v_q\}$
and at most one vertex in $B_0$, hence $d(u)\leq t$ for each vertex $u$ in $B_0\setminus\{v_p,v_q\}$. Thus $\Delta(G)\leq t$. But $d(v_0)\geq t$, which
means that $\Delta(G)=t$.

Notice that $B_0$ is a $2$-limited packing of $G$ with
$|B_0|=n-|A_0|+2=n+1-\Delta(G)$, then $L_2(G)\geq n+1-\Delta(G)$.
We complete the proof.
\end{proof}

\begin{cor}\label{coro1}
Let $G$ be a $d$-regular graph of order $n$ such that $L_k(G)= n+k-1-d$,
where $k\leq d$. Then $d\geq \frac{n}{2}$.
\end{cor}

\begin{proof}
If $d=n-1$, then $G$ is a complete graph with $L_k(G)=k$ for
$n\geq k+1\geq 2$, and the result follows from $d=n-1\geq \frac{n}{2}$.
Thus, we may assume that $d\leq n-2$. Suppose that $L_k(G)= n+k-1-d$.
Let $B$ be a maximum $k$-limited packing of $G$ with $|B|=n+k-1-d$,
and $v$ be a vertex of $G$. Since $|N[v]\cap B|\leq k$, it follows that
$|N[v]\cap \overline{B}|\geq d+1-k$. Assume that
$|N[v]\cap \overline{B}|> d+1-k$. Then $|B|<n-(d+1-k)=n+ k-1-d$,
a contradiction. Thus, there exist exactly $d+1-k$ vertices, say
$v_1, \ldots, v_{d+1-k}$, in $N[v]\cap \overline{B}$, furthermore,
$\overline{B}=\{v_1, \ldots, v_{d+1-k}\}$. Let $U=V(G)\setminus N[v]$.
Since $d\leq n-2$, it follows that $|U|> 0$. Observe that $U\subseteq B$.
Consider a vertex $u_i$ in $U$, there exist at most $k-1$ neighbours in $B$,
therefore $u_i$ is adjacent to at least $d-(k-1)$ vertices in $\overline{B}$.
But $|\overline{B}|=d+1-k$, it follows that for each vertex $u_i$ in $U$,
$u_i$ is adjacent to all the vertices in $\overline{B}$.
That is, each vertex $v_i$ in $\overline{B}$ is adjacent to
all the $n-d-1$ vertices in $U$. Note that $d(v_i)=d$ and $v_i$ has at least
one neighbour in $N[v]$, it follows that $n-d-1+1\leq d$.
Thus, we obtain that $d\geq \frac{n}{2}$.
\end{proof}

To end this section, we present the tight Nordhaus-Gaddum-type result for
$k$-limited packing numbers of graphs $G$ and $\overline{G}$ for $k\geq1$.
We first establish the tight Nordhaus-Gaddum-type lower bound for this parameter,
and characterize all the graphs obtaining this lower bound.

\begin{pro}\label{prop6}
If $G$ is a graph of order at least $k$, then $L_k(G)+L_k(\overline{G})\geq 2k$.
Moreover, $L_k(G)+L_k(\overline{G})=2k$ if and only if $G$ has one of the following
properties:

$(i)$ $G$ has exactly $k$ vertices,

$(ii)$ for each $(k+1)$-subset $X$ of $V(G)$,
$G[X]$ has maximum degree $k$ and there is a vertex outside $X$ such
that it is not adjacent to any vertex of $X$,
or
there is a vertex outside $X$ such that it is adjacent to all the vertices
of $X$ and $G[X]$ has an isolated vertex,
or
there are one vertex outside $X$ such that it is adjacent to all the vertices
of $X$ and another vertex outside $X$ such that it is not adjacent to any
vertex of $X$.
\end{pro}

\begin{proof}
Since it is impossible that $\Delta(G)=\Delta(\overline {G})=k$
for $|V(G)|=k+1$, it follows from Proposition \ref{prop2} that $L_k(G)+L_k(\overline{G})>2k$ for $|V(G)|=k+1$. And observe that
it is also impossible that for some $(k+1)$-subset $X$ of
$V(G)$, both $G[X]$ and $\overline{G}[X]$ has maximum degree $k$.
Thus, the result follows from Theorem \ref{th1}.
\end{proof}

The tight Nordhaus-Gaddum-type upper bounds for $k$-limited packing
numbers of graphs $G$ and $\overline{G}$ in the following theorem
are a generalization of Lemma \ref{lem2}.

\begin{thm}\label{th10}
Let $G$ be a graph of order $n$. Then
\begin{equation*}
L_k(G)+L_k(\overline{G})\leq
\left\{
  \begin{array}{ll}
    2n & \hbox{ if  $k\geq \max\{\Delta(G),\Delta(\overline {G})\}+1$,}\\
    n+2k-2 & \hbox{ if $ k \leq \min\{\Delta(G),\Delta(\overline {G})\}$, }\\
    2n-1 & \hbox{ otherwise. } \\
    \end{array}
\right.
\end{equation*}
Moreover, the upper bounds are tight.
\end{thm}

\begin{proof}
Suppose that $k\geq \max\{\Delta(G),\Delta(\overline {G})\}+1$. It is clear that
$L_k(G)+L_k(\overline{G})=n+n=2n$.

Suppose that
$\max\{\Delta(G),\Delta(\overline {G})\}+1> k \geq \min\{\Delta(G),\Delta(\overline {G})\}+1$. Without loss of generality, we assume that
$\Delta(\overline {G})+1> k\geq \Delta(G)+1$. It follows
that $L_k(G)=n$ and $L_k(\overline{G})<n$.
Therefore, $L_k(G)+L_k(\overline{G})\leq 2n-1$.
To show that the upper bound is tight. Let $G$ be a graph of order $k+1$ with
$\Delta(G)<k$ such that $G$ has isolated vertices. By Proposition \ref{prop2},
$L_k(G)+L_k(\overline{G})=2n-1$.

It remains to consider the case when $ k \leq \min\{\Delta(G),\Delta(\overline {G})\}$.
By Theorem \ref{th3}, we have
$L_k(G)\leq n+k-1-\Delta(G)$ and $L_k(\overline{G})\leq n+k-1-\Delta(\overline{G})$.
Thus,
\begin{eqnarray*}
L_k(G)+L_k(\overline{G}) &\leq& (n+k-1-\Delta(G))+ (n+k-1-\Delta(\overline{G}))\\
&=& 2n+2k-2-(\Delta(G)+\Delta(\overline{G}))\\
&\leq& 2n+2k-2-(\Delta(G)+\delta(\overline{G}))\\
&=& 2n+2k-2-(n-1)\\
&=& n+2k-1.
\end{eqnarray*}
Next, we claim that it is impossible that $L_k(G)+L_k(\overline{G})= n+2k-1$.
Assume to the contrary that $L_k(G)+L_k(\overline{G})= n+2k-1$.
It follows that both $G$ and $\overline{G}$ are regular graphs with
$ L_k(G)= n+k-1-\Delta(G)$ and $ L_k(\overline{G})= n+k-1-\Delta(\overline{G})$. Notice that $G$ is a $\Delta(G)$-regular graph and $\overline{G}$ is a $\Delta(\overline{G})$-regular graph, then $\Delta(G)+\Delta(\overline{G})=n-1$. Since $ L_k(G)= n+k-1-\Delta(G)$ and $ L_k(\overline{G})= n+k-1-\Delta(\overline{G})$, it follows from Corollary \ref{coro1} that $\Delta(G)\geq \frac{n}{2}$ and $\Delta(\overline{G})\geq \frac{n}{2}$,
which implies that $\Delta(G)+\Delta(\overline{G})> n-1$,
which is a contradiction. Thus, $L_k(G)+L_k(\overline{G})\leq n+2k-2$.
The following examples show that the upper bound is best possible.
Let $G=K_n-e$, where $e$ is an edge of $K_n$ and $n\geq 3$.
Then $L_1(G)= 1$ by Theorem \ref{th1}. On the other side,
$\overline{G}=K_2\cup (n-2)K_1$, then $L_1(\overline{G})= n-1$.
It is obtained that $\min\{\Delta(G),\Delta(\overline {G})\}\geq 1$
and $L_1(G)+L_1(\overline{G})= n+2k-2=n$.
\end{proof}

\section{ $2$-limited packing}

In \cite{GGHR}, the authors bounded the $2$-limited packing number for
a graph in terms of its order.

\begin{lem}{\upshape\cite{GGHR}}\label{lem26}
If $G$ is a connected graph with $|V(G)|\geq3$,
then $L_2(G)\leq\frac{4}{5}|V(G)|$.
\end{lem}

Furthermore, they imposed constraints on the minimum degree of $G$,
and obtained the following result.

\begin{lem}{\upshape\cite{GGHR}}\label{le1}
If $G$ is a connected graph, and $\delta(G)\geq k$,
then $L_k(G)\leq\frac{k}{k+1}|V(G)|$.
\end{lem}

By Lemma \ref{le1}, we have $L_2(G)\leq\frac{2}{3}|V(G)|$ for graphs with
$\delta(G)\geq 2$. It is known that trees are graphs with minimum degree $1$.
We find a class of trees $T$ with $2$-limited packing number at most
$\frac{2}{3}|V(T)|$. The minimum degree of a graph $G$ taken over all non-leaf
vertices is denoted by $\delta'(G)$.

\begin{thm}\label{thm30}
If $T$ is a tree with $\delta'(T)\geq 4$, then
$L_2(T)\leq\frac{2}{3}|V(T)|$.
\end{thm}

\begin{proof}
Since $\delta'(T)\geq 4$, it follows that $|V(T)|\geq 5$. Let $B$ be a
maximum $2$-limited packing of $T.$ By induction on the order of $T$.
If $|V(T)|=5$, then $T=K_{1,4}$, and hence $L_2(T)=2\leq \frac{2}{3}|V(T)|$
by Lemma \ref{lem3}. Let $T$ be a tree of order at least $6$. It is known that
$T$ can be regarded as a rooted tree. Take a leaf vertex $v$ of $T$, which is
the lowest level in the rooted tree $T$. Let $v_0$ be the unique neighbour of
$v$ in $T$, and $L_0$ be the set of leaf vertices in $N(v_0)$. Since $v_0$ is
adjacent to at least three leaf vertices, we have $|L_0|\geq 3$. Let $T_0$ be
the subtree obtained from $T$ by deleting all the vertices of $L_0$. By the
inductive hypothesis, $L_2(T_0)\leq \frac{2|V(T_0)|}{3}\leq \frac{2(|V(T)|-3)}{3}=\frac{2}{3}|V(T)|-2.$
Since $|N[v_0]\cap B|\leq 2$, it follows that $|L_0\cap B|\leq 2$.
Hence, $L_2(T)\leq L_2(T_0)+2=\frac{2}{3}|V(T)|$.
\end{proof}

It is shown that both the opening packing number and the $1$-limited
packing number of a graph with diameter at most $2$ are small
in Remark $1$ and Lemma \ref{lem16}. These results naturally
lead to the following problem: can the $2$-limited packing number of
a graph $G$ be bounded by a constant for $diam(G)\leq 2$? It is known
that the graph with order $n$ and diameter $1$, which is exactly $K_n$,
has $2$-limited packing number $2$ by Lemma \ref{lem3}. Thus, we only
need to investigate the $2$-limited packing number of graphs with
diameter $2$. Theorem \ref{th6} answers the above question.

\begin{thm}\label{th6}
For any positive integer $a$ with $a\geq 2$, there exists a graph $G$
with $diam(G)=2$ such that $L_2(G)=a.$
\end{thm}

\begin{proof}
We construct a graph $G$ with $diam(G)=2$ such that $L_2(G)=a$
for $a\geq 2$ as follows.

First, suppose that $X=\{x_1,x_2,\ldots, x_a\}$ and
$Y=\{y_1,y_2,\ldots, y_{\frac{a(a-1)}{2}}\}$ with $X\cap Y=\emptyset$.
Let $G$ be a graph with $V(G)=X\cup Y$ such that $G[X]$ consists of
$a$ isolated vertices, $G[Y]$ is a clique and each pair of distinct
vertices in $X$ has a unique common neighbour in $Y$. Obviously,
it is true that $diam(G)=2$. Now we need to show that $L_2(G)=a$.
Notice that $|V(G)|=a+\frac{a(a-1)}{2}$ and $\Delta(G)=\frac{a(a-1)}{2}+1$,
then $L_2(G)\leq |V(G)|+1-\Delta(G)=a$ by Corollary \ref{cor14}.
Observe that $X$ is a $2$-limited packing of $G$, thus $L_2(G)= a$.
\end{proof}

But we can find graphs $G$ with $diam(G)=2$ such that $L_2(G)$ is small.
First, we give some auxiliary lemmas.

\begin{lem}{\upshape\cite{JP}}\label{lem20}
Every planar graph of diameter $2$ has domination number at most $2$
except for the graph $F$ of Fig. $1$ which has domination number $3$.
\end{lem}
\begin{figure}[htbp]
\centering
\includegraphics[width=0.2\textwidth]{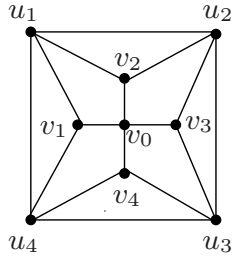}
\caption{A counterexample $F$ of Lemma \ref{lem20}}
\end{figure}

\begin{lem}\label{lem21}
If $G$ is a graph with order $n$ and $\Delta(G)=n-1$, then $L_2(G)=2$.
\end{lem}

\begin{proof}
By Lemma \ref{lem7}, we have
$L_2(G)\geq \lceil \frac{2diam(G)+2}{3}\rceil \geq \lceil\frac{4}{3}\rceil=2$.
On the other hand, $L_2(G)\leq n-\Delta(G)+1=2$ by Corollary \ref{cor14}.
Thus, $L_2(G)=2$.
\end{proof}

\begin{lem}\label{lem22}
Let $G$ be a graph with diameter $2$. Then

$(i)$ if $G$ has a cut vertex, then $L_2(G)=2$,

$(ii)$ if $G$ is a planar graph, then $L_2(G)\leq 4$.
\end{lem}

\begin{proof}
Firstly we prove part $(i)$. Let $v_0$ be a cut vertex of $G$.
We claim that for any vertex $u\in V(G)\setminus \{v_0\}$, $d(u,v_0)=1$.
Suppose that there exists a vertex $u_0\in V(G)\setminus \{v_0\}$
such that $d(u_0,v_0)=2$. Let $w_0$ be another vertex of $G$ such that
$w_0$ and $u_0$ are contained in different components of $G-v_0$. Then
$d(u_0,w_0)\geq 3$, which contradicts to $diam(G)=2$. Thus,
for any vertex $u\in V(G)\setminus \{v_0\}$, $d(u,v_0)=1$. It is obtained
that $d(v_0)=|V(G)|-1$, hence $L_2(G)=2$ by Lemma \ref{lem21}.

Next we prove part $(ii)$. Let $G$ be a planar graph with diameter $2$.
Then it follows from Lemma \ref{lem4} and Lemma \ref{lem20} that
$L_2(G)\leq 2\gamma(G)\leq 4$ except for the graph $F$ of Fig. $1$.
It remains to verify the graph $F$ in Fig. $1$. Let $B$ be a maximum
$2$-limited packing of $F$. Observe that $|\{u_1,\ldots,u_4\}\cap B|\leq 2$
and $|\{v_0,\ldots,v_4\}\cap B|\leq 2$, it follows that $L_2(F)\leq 4$.
\end{proof}

Next, we get the better upper bound of the $2$-limited
packing number of graphs with large diameter.

\begin{thm}\label{th9}
If $G$ is a connected graph of order $n$, then
$L_2(G)\leq n+1-\Delta(G)-\lfloor \frac{diam(G)-4}{3} \rfloor$.
\end{thm}

\begin{proof}
If $diam(G)\leq 2$, then $L_2(G)\leq |V(G)|+1-\Delta(G)\leq n+1-\Delta(G)
-\lfloor \frac{diam(G)-4}{3} \rfloor $ by Corollary \ref{cor14}.
Thus, we may assume that $diam(G)\geq 3$ in the following.
Let $B$ be a maximum $2$-limited packing in $G$, and $u$ be a vertex
of degree $\Delta(G)$. Then $|N[u]\cap B|\leq 2$. Let $P$ be a path
of length $diam(G)$ between $x$ and $y$ in $G$. We claim that
$|V(P)\cap N[u]|\leq3$, otherwise using the same argument in proof of
Theorem \ref{th5}, we can find a path between $x$ and $y$, whose length
is less than $diam(G)$, a contradiction. It is also obtained that
$|\{x,y\}\cap N[u]|\leq 1$, otherwise $d(x,y)\leq 2$, which contradicts
to $diam(G)\geq 3$. Without loss of generality, assume that
$x\in V(G)\setminus N[u]$.

\textbf{Case 1.} $V(P)\cap N[u]=\emptyset$.

By Lemma \ref{lem3},
$P$ has at most $\lceil\frac{2|V(P)|}{3}\rceil $ vertices in
$B$. Then $|V(P)\cap \overline {B}|\geq \lfloor\frac{|V(P)|}{3}\rfloor$.
On the other hand, $|N[u]\cap \overline {B}|\geq \Delta(G)+1-2=\Delta(G)-1$.
Thus,
\begin{eqnarray}
 | \overline {B}|
          &\geq & \Delta(G)-1+\lfloor\frac{|V(P)|}{3}\rfloor \nonumber\\
          &=&     \Delta(G)-1+\lfloor \frac{diam(G)+1}{3} \rfloor\nonumber.
\end{eqnarray}

\textbf{Case 2.}  $V(P)\cap N[u]\neq\emptyset$.

Let $P_x$, $P_y$ be the paths obtained from $P$ by deleting all
the vertices in $N[u]$ such that $P_x$ and $P_y$ contain $x$, $y$,
respectively. Let $H=P_x\cup P_y$.
It is worth mentioning that $V(P_y)= \emptyset$ if $y \in N[u]$.
Observe that $|V(H)|=|V(P_x)|+|V(P_y)|\geq |V(P)|-3 = diam(G)-2$.
Since $H$ has at most
$\lceil\frac{2|V(P_x)|}{3}\rceil+\lceil\frac{2|V(P_y)|}{3} \rceil$
vertices in $B$ by Lemma \ref{lem3}, it follows that
$|V(H)\cap \overline{B}|\geq \lfloor\frac{|V(P_x)|}{3}\rfloor+\lfloor\frac{|V(P_y)|}{3} \rfloor$.
Since $|N[u]\cap \overline {B}|\geq \Delta(G)+1-2=\Delta(G)-1$ and
$V(H)\cap N[u]=\emptyset$, we have
\begin{eqnarray}
 | \overline{B}|
          &\geq & \Delta(G)-1+\lfloor\frac{|V(P_x)|}{3}\rfloor+\lfloor\frac{|V(P_y)|}{3} \rfloor \nonumber\\
          &\geq&  \Delta(G)-1+\lfloor \frac{|V(H)|-2}{3} \rfloor\nonumber\\
          &\geq&     \Delta(G)-1+\lfloor \frac{diam(G)-4}{3} \rfloor
\end{eqnarray}
Combining Case $1$ and Case $2$, we have
$| \overline{B}|\geq \Delta(G)-1+\lfloor \frac{diam(G)-4}{3} \rfloor$.
Hence, $L_2(G)=|B|\leq n+1-\Delta(G)-\lfloor \frac{diam(G)-4}{3} \rfloor$.

\noindent\textbf{Remark 4.}
The upper bound in Theorem \ref{th9} is better than
that in Corollary \ref{cor14} for $diam(G)\geq 7$.
\end{proof}

\section{ Comparing $L_2(T)$ with $L_1(T)$ and $\rho^{0}(T)$}

In this section, we study the relationship among the $2$-limited
packing number, the $1$-limited packing number and the open packing
number of trees.

\begin{lem}{\upshape\cite{HHS}}\label{lem27}
 For any graph $G$, $L_1(G)\leq \rho^{0}(G)\leq 2L_1(G)$.
\end{lem}

Similarly, we consider the relationship between the $2$-limited
packing number and the $1$-limited packing number of graphs.

\begin{pro}\label{prop4}
For any graph $G$ with edges,
$L_1(G)+1\leq L_2(G) \leq \frac{2(\Delta(G)^2+1)}{\delta(G)+1}L_1(G).$
\end{pro}

\begin{proof}
The Lower bound is evidently true for $\Delta(G)\geq 1$ by Lemma \ref{lem5}
and Remark $2$. It remains to verify the upper bound. Combining
Lemma \ref{lem31} and Lemma \ref{lem10}, we have
\begin{eqnarray*}
L_1(G) &\geq& \frac{|V(G)|}{\Delta(G)^2+1}\\
&=& \frac{2|V(G)|}{\delta(G)+1}\frac{\delta(G)+1}{2(\Delta(G)^2+1)}\\
&\geq&  L_2(G)\frac{\delta(G)+1}{2(\Delta(G)^2+1)}.
\end{eqnarray*}
That is, $L_2(G) \leq \frac{2(\Delta(G)^2+1)}{\delta(G)+1}L_1(G).$
\end{proof}

With respect to trees, the above result can be further improved.
Recall that a star is a tree with diameter at most $2$.
Define a $t$-spider to be a tree obtained from a star by subdividing
$t$ of its edges once.

\begin{thm}\label{th11}
For any tree $T$, $L_1(T)+1\leq L_2(T)\leq 2L_1(T)$. Moreover,
$L_1(T)+1= L_2(T)$ if and only if $T$ is a $t$-spider with
$0\leq t<\Delta(T)$.

\end{thm}

\begin{proof}
The lower bound holds from $\Delta(T)\geq 1$ by Lemma \ref{lem5},
and the upper bound follows from $L_2(T)\leq 2\gamma (T)=2L_1(T)$
by Lemma \ref{lem34} and Lemma \ref{lem4}.

Next, we show that $L_1(T)+1= L_2(T)$ if and only if
$T$ is a $t$-spider with $0\leq t<\Delta(T)$. Let $T$
be a $t$-spider with
$V(T)=\{r,v_1,\ldots,v_t, w_1,\ldots,w_t, u_1,\ldots,u_s\}$
and $E(T)=\{rv_1,\ldots,rv_t\}\cup \{ru_1,\ldots,ru_s\} \cup \{v_1w_1,\ldots,v_tw_t\}$, where $s\geq1$ and $t\geq 0$.
If $t=0$, then $T$ is a star and the results follows from
$L_2(T)=2$ and $L_1(T)=1$. Now we assume that $t\geq1$.
Notice that $|V(T)|=1+2t+s$ and $\Delta(T)=t+s$.
By Corollary \ref{cor14}, we have
$L_2(T)\leq |V(T)|+1-\Delta(T)=t+2$. Observe that
$\{ v_1,u_1, w_1,\ldots,w_t\}$ is a $2$-limited packing of $T$,
it follows that $L_2(T)= t+2$. On the other side, it follows from
Theorem \ref{th3} that $L_1(T)\leq |V(T)|-\Delta(T)=t+1$. And it
is clear that $\{w_1,\ldots,w_t, u_1\}$ is a $1$-limited packing
of $T$, so $L_1(T)= t+1$. Thus, $L_1(T)+1= L_2(T)$.

It remains to show the converse. Let $T$ be a tree with
$L_1(T)+1= L_2(T)$. We first give the following claim.

\textbf{Claim 1.} $diam(T)\leq 4$.

\noindent\textbf{Proof of Claim 1:} Assume to the contrary that there
is a path $Q=v_1\cdots v_6$ of order $6$ in $T$.
Let $B_1$ be a maximum $1$-limited packing of $T$. By Lemma \ref{lem3},
it is obtained that $|V(Q)\cap B_1|\leq 2$. To have a contradiction,
we aim to find a $2$-limited packing that contains $|B_1|+2$ vertices
in $T$. Suppose that $V(Q)\cap B_1=\emptyset$. We claim that
$B_2=B_1\cup\{v_1,v_6\}$ is a $2$-limited packing of $T$.
It is clear that $|N[v_i]\cap B_2|\leq 2$ for each vertex $v_i$
on $Q$. Consider each vertex $u$ outside $Q$. First, we know
$|N[u]\cap B_1|\leq 1$ for each vertex $u$ outside $Q$. Since $T$
is a tree and has no cycle, it follows that each vertex $u$
outside $Q$ has at most one neighbour on $Q$, which means
$|N[u]\cap \{v_1,v_6\}|\leq 1$. Thus, $|N[u]\cap B_2|\leq 2$ for
each vertex $u$ outside $Q$. As a result, we obtain that
$B_2=B_1\cup\{v_1,v_6\}$ is a $2$-limited packing of $T$.
Suppose that $V(Q)\cap B_1=\{v_i\}$ for some $1\leq i\leq 6$.
If $i=1$, then $B_1\cup\{v_2,v_6\}$ is a $2$-limited packing of $T$.
It is worth mentioning that $v_2$ is not adjacent to any vertex in $B_1\setminus\{v_1\}$, otherwise $|N[v_2]\cap B_1|\geq2$, a contradiction.
If $2\leq i\leq5$, then $B_1\cup\{v_1,v_6\}$ is a $2$-limited packing of $T$.
If $i=6$, then $B_1\cup\{v_1,v_5\}$ is a $2$-limited packing of $T$.
Suppose that $V(Q)\cap B_1=\{v_i,v_j\}$ for some $1\leq i\neq j\leq 6$.
If $(i,j)\in\{(1,4),(3,6)\}$, then $ B_1\cup\{v_2,v_5\}$ is a $2$-limited
packing of $T$. If $(i,j)\in\{(1,5),(1,6),(2,5),(2,6)\}$, then
$(B_1\setminus \{v_i,v_j\})\cup\{v_1,v_2,v_5,v_6\}$ is a $2$-limited
packing of $T$. By the above argument, we have $L_2(T)\geq L_1(T)+2$,
which is a contradiction. Thus, it is obtained that $diam(T)\leq4$.

Suppose that $diam(T)\leq4$. Let $F$ be a tree with diameter $4$
and a unique vertex $f_0$ of maximum degree $3$ such that $f_0$ is
adjacent to two leaf vertices. Let $B_1$ be a maximum $1$-limited
packing of $T$. It is clear that $|V(F)\cap B_1|\leq 2$. We claim
that $T$ has no $F$ as a subtree. Suppose to the contrary that
$F$ is a subtree of $T$. By the similar argument when $T$ contains
a path $P_6$, we always find a $2$-limited packing with $L_1(T)+2$
vertices in $T$ as depicted in Fig. 2, which is a contradiction.
Thus, $T$ has no $F$ as a subtree, which implies that $T$ is a
$t$-spider with $0\leq t\leq \Delta(T)$. Notice that if $\Delta(T)=1$,
then it is clear that $t=0$. For $\Delta(T)\geq2$, we claim that
$t< \Delta(T)$. Assume to the contrary that $t=\Delta(T)\geq 2$.
Let $r$ be a vertex of maximum degree $t$ with $N(r)=\{v_1,\ldots,v_t\}$
and $N(v_i)=\{r,w_i\}$ for $1\leq i\leq t$, where $w_1,\ldots,w_t$
are $t$ leaf vertices of $T$. Observe that $\{v_1, v_2, w_1,\ldots,w_t\}$
is $2$-limited packing of $T$, it follows that $L_2(T)\geq t+2$.
On the other hand, $\{r, v_i, w_i\}$ has at most one vertex in a
$1$-limited packing of $T$ for each $1\leq i\leq t$, it follows that
$L_1(T)\leq t$. Thus, $L_2(T)\geq L_1(T)+2$, which is a contradiction.
As a result, $T$ is a $t$-spider with $0\leq t< \Delta(T)$.

\begin{figure}[!htbp]
\centering
\includegraphics[width=0.6\textwidth]{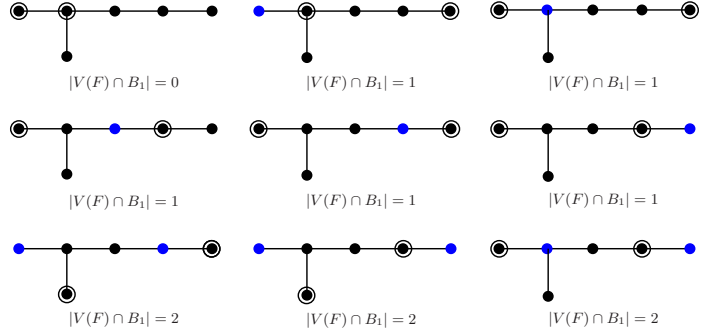}
\caption{The three cases can arise on the number of
$|V(F)\cap B_1|$. In each case, the blue points correspond
to the vertices in $B_1$, the black points correspond to
the vertices outside $B_1$, and the circled black points correspond
to the vertices outside $B_1$ that will be added into $B_2$.}
\end{figure}

\end{proof}

\noindent\textbf{Remark 5.}
By the proof of Theorem \ref{th11}, we know that $L_2(T)= 2L_1(T)$
if and only if $L_2(T)= 2\gamma (T)$. And all the trees $T$ with
$L_2(T)= 2\gamma (T)$ are characterized in \cite{GGHR}.

\vskip 0.3cm

Similarly, we compare the $2$-limited packing number with the open packing
number of trees. We first define a class of trees, which is needed in the
following theorem. Let $\mathcal{T}$ be the set of trees $T$ whose vertex
set can be partitioned into two disjoint subsets $S_0$ and $R_0$, satisfying
the following properties:

$(i)$ $T[S_0]=aK_2$, and each copy of $K_2$ has at least one vertex
with degree $1$ in $T$, where $a$ is an positive integer,

$(ii)$ for each $r\in R_0$, $|N(r)\cap S_0|=1$.

\begin{thm}\label{th12}
For any tree $T$, $\rho^{0}(T)\leq L_2(T)\leq2\rho^{0}(T)$.
Moreover, $\rho^{0}(T)= L_2(T)$ if and only if
$T\in \mathcal{T}$.
\end{thm}

\begin{proof}
For the lower bound, we give the stronger result that
$ L_2(G)\geq \rho^{0}(G)$ for any graph $G$. Let $S$ be an
open packing of $G$ with $|S|=\rho^{0}(G)$. Then
$|N(v)\cap S|\leq 1$ for each vertex $v$ of $G$. Obviously,
$|N[v]\cap S|\leq 2$ for each vertex $v$ of $G$. It is obtained
that $S$ is a $2$-limited packing of $G$, therefore
$\rho^{0}(G)\leq L_2(G)$. On the other hand, since
$\gamma(T)\leq \gamma_t(T)$, it follows that
$L_2(T)\leq2\gamma(T)\leq 2\gamma_t(T)=2\rho^{0}(T)$
by Lemma \ref{lem33} and Lemma \ref{lem4}.

Next, we show that $\rho^{0}(T)= L_2(T)$ if and only if
$T\in \mathcal{T}$. Let $T$ be a tree in $\mathcal{T}$.
If $T$ has only two vertices, then $T=K_2$, and hence the
result trivially holds. Now we assume that $|V(T)|\geq 3$.
Observe that $S_0$ is an open packing of $T$, if follows that
$\rho^{0}(T)\geq 2a$. By the definition of $T$, it is obtained
that $V(T)$ can be partitioned into $V_1\cup \cdots\cup V_a$
such that $G[V_i]$ is a star for each $1\leq i\leq a$.
Notice that $V_i$ has at most two vertices in a $2$-limited
packing of $T$ for each $1\leq i\leq a$, then $L_2(T)\leq 2a$.
Since $\rho^{0}(T)\leq L_2(T)$, it follows that
$\rho^{0}(T)= L_2(T)$ for each tree $T$ in $\mathcal{T}$.

Conversely, suppose that $\rho^0(T)= L_2(T)$. Let $S$ be a maximum
open packing of $T$. It is known that each component of $T[S]$ is
$K_1$ or $K_2$. To show $T\in \mathcal{T}$, we give the following
claims.

\textbf{Claim 1.} $T[S]=tK_2$.

\noindent\textbf{Proof of Claim 1:} Suppose to the contrary
that there is at least one isolated vertex, say $v$, in $T[S]$.
Since $T$ is connected, it follows that $v$ has a neighbour,
say $r$, in $\overline {S}$. Let $B=S\cup\{r\}$. Next we show
that $B$ is a $2$-limited packing of $T$. Since $r$ is not
adjacent to any vertex in $S\setminus\{v\}$, it follows that for
each vertex $v$ in $B$, $|N[v]\cap B|\leq 2$. On the other hand,
for each vertex $u$ in $\overline{B}$, we have $|N[u]\cap S|\leq 1$,
and hence $|N[u]\cap B|\leq 2$. Thus, $B$ is a $2$-limited packing
of $T$, and $L_2(T)\geq |B|=|S|+1$, which is a contradiction. Thus,
there is no isolated vertex in $T[S]$, so $T[S]=tK_2$.

\textbf{Claim 2.} For each $r\in \overline{S}$, $|N(r)\cap S|=1$.

\noindent\textbf{Proof of Claim 2:} By the definition of the open
packing, we know $|N(r)\cap S|\leq 1$ for any vertex $r\in \overline{S}$.
To show this claim, it remains to prove $|N(r)\cap S|\geq 1$ for any vertex
$r\in \overline{S}$. Assume that there is a vertex $r_0\in \overline{S}$
such that $N(r_0)\cap S=\emptyset$. Then $S\cup \{r_0\}$ is a $2$-limited
packing of $T$, so $L_2(T)\geq \rho^{0}(T)+1$, a contradiction. Hence,
we have $|N(r)\cap S|=1$ for each $r\in \overline{S}$.

\textbf{Claim 3.} Each component of $T[S]$ has at least one vertex
with degree $1$ in $T$.

\noindent\textbf{Proof of Claim 3:} Suppose that $T[S]$ has one component $K_2=v_1v_2$, where $d(v_i)\geq2$ for $i=1,2$. It is obtained that there
is a path $P=uv_1v_2w$ in $T$, where $u,w\in \overline{S}$. Notice that
each vertex on $P$ has the property that its neighbours outside $P$ are
not contained in $S$, otherwise there is a vertex on $P$ such that it has
at least two neighbours in $S$, which is a contradiction. It is obtained that
$(S\setminus \{v_2\})\cup\{u,w\}$ is a $2$-limited packing of $T$,
which means $L_2(T)\geq \rho^0(T)+1$, which is a contradiction. Thus,
each component of $T[S]$ has at least one vertex with degree $1$ in $T$.

By the above claims, we get that if $\rho^0(T)= L_2(T)$, then
$T\in \mathcal{T}$ with $S=S_0$, this completes the proof.
\end{proof}

Graphs with $\rho^{0}(G)=1$, graphs with $L_k(G)=k$ for $k=1,2$ are
characterized in Lemma \ref{lem15} and Theorem \ref{th1}, respectively.
So we assume that $a\geq 2$ in the following theorem.

\begin{thm}\label{th8}
For each pair of integers $a$ and $b$ with $a\geq 2$ and
$a+1\leq b\leq 2a$, there exists a tree $T$ such that
$\rho^{0}(T)=L_1(T)=a$ and $L_2(T)=b$.
\end{thm}

\begin{proof}
Suppose $a$ and $b$ are two positive integers with $a+1\leq b\leq 2a$.
Let $b=a+r$ with $1\leq r\leq a$. To construct a tree $T$ with
$\rho^{0}(T)=L_1(T)=a$ and $L_2(T)=a+r$ for $a\geq2$ and $1\leq r\leq a$,
we distinguish the following two cases.

\textbf{Case 1.} $a=r$.

Suppose that $Q_i=x_iy_iz_i$ is a path of order $3$ for
$1\leq i\leq a$. Let $T$ be the tree obtained from
$Q_1\cup\cdots \cup Q_a$ by adding the edge $y_iy_{i+1}$ for
$1 \leq i\leq a-1$. First, we show that $L_2(T)=a$. Since
each $V(Q_i)$ has at most two vertices in a $2$-limited packing
of $T$ for $1\leq i\leq a$, we have $L_2(T)\leq 2a$. It is observed
that $\{x_1, \cdots, x_a, z_1,\cdots,z_a\}$ is a $2$-limited packing
of $T$, thus $L_2(T)=2a$. Next, we show that $\rho^{0}(T)=L_1(T)=a$.
Note that $L_1(T)\leq \rho^{0}(T)$ by Lemma \ref{lem27}, then it
is sufficient to show that $L_1(T)\geq a$ and $\rho^{0}(T)\leq a$.
Obviously, $\{x_1, \cdots, x_a\}$ is a $1$-limited packing
of $T$, thus $L_1(T)\geq a$. It remains to show that $\rho^{0}(T)\leq a$.
It is clear that $\{y_1, \cdots, y_a\}$ is a total dominating set of $T$,
which implies that $\gamma_t(T)\leq a$. By Lemma \ref{lem33}, we have $\rho^{0}(T)=\gamma_t(T)\leq a$.

\textbf{Case 2.} $1\leq r\leq a-1$.

Consider a star $A=K_{1,a}$ with $V(A)=\{v_0,v_1,\ldots,v_a\}$ and
$d(v_0)=a$. Let $T$ be the tree obtained from $A$ by adding two pendent
edges $v_iw_i$ and $v_iw_i^{'}$ to each $v_i$ of $A$ for $1\leq i\leq r-1$,
and one pendent edge $v_iw_i$ at each $v_i$ of $A$ for $r\leq i\leq a-1$.
Fig. $3$ gives an example for $a=8$, $b=12$.

\begin{figure}[htbp]
\centering
\includegraphics[width=0.3\textwidth]{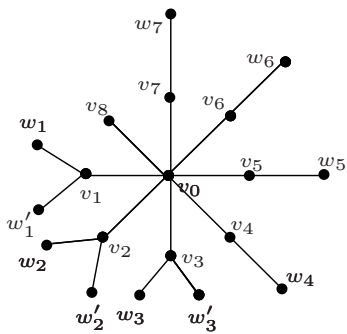}
\caption{A graph with $\rho^{0}(T)=L_1(T)=8$ and $L_2(T)=12$}
\end{figure}

To obtain that $L_1(T)=\rho^{0}(T)=a$, it suffices to prove that
$L_1(T)\geq a$ and $\rho^{0}(T)\leq a$ by Lemma \ref{lem27}.
Obviously, $\{w_1,\ldots,w_{a-1}, v_a\}$ is a $1$-limited
packing of $T$, so $L_1(T)\geq a$.
On the other hand, let $S$ be a maximum open packing of $T$ with
$|S|=\rho^{0}(T)$. It only need to show $|S|\leq a$.
Suppose that $v_0\in S$. It follows that
$\{v_i: 1\leq i\leq a\}$ has at most one vertex in $S$, and
$\{w_i,w_j': 1\leq i\leq a-1, 1\leq j\leq r-1\}$ has no vertex in $S$.
It is obtained that $S=\{v_0, v_i\}$ for some $i$ with $1\leq i\leq a$,
and hence $|S|=2\leq a$. Suppose that $v_0 \notin S$. If $v_a\in S$, then
$\{v_i: 1\leq i\leq a-1\}$ has no vertex in $S$ and $\{w_i,w_i'\}$
has at most one vertex in $S$ for each $1\leq i\leq a-1$, and hence
$|S|\leq a$. If $v_a\notin S$, then both $\{v_i: 1\leq i\leq a-1\}$ and
$\{w_i,w_i'\}$ for each $1\leq i\leq a-1$ have at most one vertex in $S$, so
$|S|\leq a$.

It remains to show that $L_2(T)= a+r$ with $1\leq r\leq a-1$.
Note that $T$ has $2a+r-1$ vertices and $\Delta(T)=a$.
By Corollary \ref{cor14}, we have $L_2(T)\leq |V(T)|+1-\Delta(T)= a+r$.
Observe that $\{v_a,v_{a-1}\}\cup\{w_i,w_j': 1\leq i\leq a-1, 1\leq j\leq r-1\}$
is a $2$-limited packing of $T$, it follows that $L_2(T)= a+r$.
\end{proof}

\end{document}